\documentclass[11pt]{amsart}
\usepackage{amsmath,amssymb,latexsym,soul,cite,mathrsfs}

\usepackage{color,enumitem,graphicx}
\usepackage[colorlinks=true,urlcolor=blue,
citecolor=red,linkcolor=blue,linktocpage,pdfpagelabels,
bookmarksnumbered,bookmarksopen]{hyperref}
\usepackage[english]{babel}

\usepackage[left=2.9cm,right=2.9cm,top=2.8cm,bottom=2.8cm]{geometry}
\usepackage[hyperpageref]{backref}

\usepackage[colorinlistoftodos]{todonotes}
\makeatletter
\providecommand\@dotsep{5}
\def\listtodoname{List of Todos}
\def\listoftodos{\@starttoc{tdo}\listtodoname}
\makeatother

\numberwithin{equation}{section}

\def\R {{\rm I}\hskip -0.85mm{\rm R}}

\newtheorem{theorem}{Theorem}[section]
\newtheorem{proposition}[theorem]{Proposition}
\newtheorem{lemma}[theorem]{Lemma}

\newtheorem{example}[theorem]{Example}

\newtheorem{remark}{Remark}

\title[On a generalized Kirchhoff equation]
{ On a generalized Kirchhoff equation \\ with sublinear nonlinearities}

\author[J. R. Santos Jr.]{Jo\~ao R. Santos Junior}
\author[G. Siciliano]{Gaetano Siciliano}

\address[J. R. Santos Jr.]{\newline\indent Faculdade de Matem\'atica
\newline\indent 
Instituto de Ci\^{e}ncias Exatas e Naturais
\newline\indent 
Universidade Federal do Par\'a
\newline\indent
Avenida Augusto corr\^{e}a 01, 66075-110, Bel\'em, PA, Brazil}
\email{\href{mailto: joaojunior@ufpa.br }{joaojunior@ufpa.br}}
\address[G. Siciliano]{\newline\indent Departamento de Matem\'atica
\newline\indent 
Instituto de Matem\'atica e Estat\'istica
\newline\indent 
 Universidade de S\~ao Paulo 
\newline\indent 
Rua do Mat\~ao 1010,  05508-090, S\~ao Paulo, SP, Brazil }
\email{\href{mailto:sicilian@ime.usp.br}{sicilian@ime.usp.br}}

\thanks{Jo\~ao R. Santos was partially
supportedby PROCAD/CASADINHO: 552101/2011-7, Brazil. Gaetano Siciliano  was partially supported by
Fapesp and CNPq, Brazil. }

\begin{document}

\maketitle
\begin{abstract}
In this paper we consider a generalized Kirchhoff equation in a bounded domain
 under the effect of a sublinear nonlinearity. Under suitable assumptions on the 
data of the problem we show that, with a 
 simple change of variable, the equation can be reduced to a classical semilinear equation
 and then studied with standard  tools.
\end{abstract}
\maketitle

\section{Introduction}


In this article we study the existence of solutions $u:\Omega\to \R$ for the following nonlocal problem in divergence form

\begin{equation}\label{P}\tag{P}
\left \{ \begin{array}{ll}
-\text{div}\left(m\left (u, |\nabla u|_{2}^{2}\right)\nabla u\right)= f(x, u) & \mbox{in $\Omega$,}\\
u=0 & \mbox{on $\partial\Omega$,}
\end{array}\right.
\end{equation}
where $\Omega\subset\R^{N}$ is a bounded domain with smooth boundary, 
$$m:\R\times[0,\infty) \to \R, \quad f:\Omega\times \R\to \R$$ are  given functions satisfying  suitable conditions which will be given later. 
Hereafter we denote with  $|\cdot|_{p}$  the usual $L^{p}(\Omega)-$norm. 
By a solution of the above problem, we mean a function $u_{*}\in H^{1}_{0}(\Omega)\cap L^{\infty}(\Omega)$
such that
$$\forall \varphi\in H^{1}_{0}(\Omega) : \quad
\int_{\Omega} m(u_{*}(x) , |\nabla u_{*}|_{2}^{2}) \nabla u_{*} \nabla \varphi \,dx = \int_{\Omega}f(x, u_{*})\varphi \, dx,$$
whenever the integrals make sense.

\medskip

When the function $m$ does not depend on $u$ we have the classical problem
\begin{equation*}
\left \{ \begin{array}{ll}
-m\left( |\nabla u|_{2}^{2}\right)\Delta u= f(x, u) & \mbox{in $\Omega$,}\\
u=0 & \mbox{on $\partial\Omega$,}
\end{array}\right.
\end{equation*}
 which is the $N$-dimensional  version, in the stationary case, of the 
{\sl Kirchhoff equation} introduced in \cite{kirchhoff}. We do not list here the huge amount of papers concerning this equation.
On the other hand in many physical problems, rather then on $|\nabla u|_{2}$, the function $m$
depends on the unknown $u$, or even on quantities related to $u$, as its $L^{1}$-norm (see e.g. \cite{CR}).
\medskip

\medskip

Problem  \eqref{P}  studied in this paper  can be  considered as a slight generalization of
the Kirchhoff equation.
 It  can modelize the   following physical situation.
Consider an elastic membrane with shape $\Omega$ and fixed edge $\partial\Omega$, initially in rest.  Let $f$ be a given external force acting on $\Omega$ and $u(x)$ the transverse displacement at a point $x\in \Omega$, with respect to initial position, of the equilibrium solution. When the displacement is small, the number $(1/2)\int_{\Omega}|\nabla u|^{2}dx$ give us a good approximation of the variation of the superficial area of the membrane $\Omega$. What we are assuming here is that,
 the velocity of the displacement of the membrane
 is proportional to the gradient of the displacement with a factor $m$ depending not only
 on the variation of the superficial area of $\Omega$, but even on the same displacement:
$$
\textbf v=-m\left(u, |\nabla u|^{2}_{2} \right)\nabla u.
$$
Such a model is quite reasonable and, to the best of  our knowledge, it has not been considered before. 

Beside the physical interest, the equation is also challenging from a mathematical point of view.

\medskip

In this paper we will treat the situations in which
$f=f(x)$ and $f=f(x,u)$ with sublinear growth in $u$.
However
 our results will be stated and proved in the next sections; indeed to state our theorems
 (especially Theorem \ref{th:casof(xu)}) 
some preliminaries are needed.
Here we say that the function $m$ will satisfy 
 quite general assumptions. In particular, in contrast to the case in which
$m$  depend only on $|\nabla u|_{2},$ here there is no restriction in the growth at infinity 
of $m$ with respect to  $|\nabla u|_{2}^{2}$.

The main novelty of our approach is that the proofs are based on a simple ``change of variable'' device
which seems not to have been used for these kind of nonlocal equations. 
With the use of the change of variable, the equation \eqref{P} is reduced to a ``local'' semilinear equation,
for which various tools are available to solve it. 

For other type of change of variables in this type of problems, see also \cite{ACM,AA}.

\begin{remark}
Here, based on a result in \cite{BO}, we will assume a sublinear assumption on $f$
and we will use  topological tools. However, depending on other type of assumptions on the nonlinearity $f$,
other methods, for instance variational, can be  employed to solve the local equation in which the problem is transformed after the change of variable (and then recover a solution of the original equation).
\end{remark}

\begin{remark}
 We believe that change of variable of this type can be used also to deal with other nonlocal equations.
\end{remark}

The paper is organized as follows. 

In Section \ref{se:prelim} we 
present the general approach to solve problem \eqref{P},
we state the main assumptions  on $m$, we introduce its primitive $M$ 
 and give
some of its properties (see Lemma \ref{le}).
In Section \ref{sec:f(x)} we state and prove our main result in the case $f=f(x)$, i.e. Theorem \ref{th:main1}.
In Section \ref{sec:general} we consider the general case $f=f(x,u)$: actually we give a criterion (Theorem \ref{th:th})
which ensures the existence of a solution to problem \eqref{P}. Then two applications are given 
in Section \ref{sec:particulares}.

\medskip

\section{Assumptions on $m$  and a useful Lemma}\label{se:prelim}

Our approach to treat  problem \eqref{P}, in both cases $f=f(x)$ and $f=f(x,u)$, consists in the following  steps. 
Firstly, for every fixed $r\geq0$, we consider the auxiliary problem
\begin{equation*}
\left \{ \begin{array}{ll}
-\text{div}\left(m(u,r)\nabla u\right)= f(x, u) & \mbox{in $\Omega$,}\smallskip \\
u=0 & \mbox{on $\partial\Omega$,}
\end{array}\right.
\end{equation*}
associated to \eqref{P} for which the existence of a unique solution $u_{r}$
will be guaranteed.
Secondly, we show that the map $$S: r\mapsto \int_{\Omega}|\nabla u_{r}|^{2}dx$$ 
has a fixed point, which is of course a solution of the original problem \eqref{P}.

\medskip

Before to state our basic assumptions, let us introduce the following convention:
for every $r\geq0$, let us denote with $m_{r}$ the map
$$m_{r}: t\in\R\mapsto m(t,r) \in \R.$$

We  suppose that $m: \R\times [0,\infty)\to(0,\infty)$ is a function satisfying the following conditions: \medskip
%
%
%
%

\begin{enumerate}[label=(m\arabic*),ref=m\arabic*,start=0]
\item\label{m_{0}} $m: \R\times[0,+\infty)\to \mathbb (0,+\infty)$ is a continuous function;  \medskip
\item\label{m_{1}} there is $\mathfrak m >0$ such that $m(t, r)\geq \mathfrak m $ for all $t\in\R$ and $r\in[0,\infty)$;  
 
 \medskip
 
\item\label{m_{2}} for each $r\in [0,+\infty)$ the map 
$m_{r}:\R\to(0,+\infty)$ is strictly decreasing in $(-\infty,0)$  and strictly increasing in $(0,+\infty)$. \medskip 

\end{enumerate}
The class of functions satisfying the above assumptions is very large: 
the above conditions are satisfied, for example, by
$m(t,r)= t^{2}(r^{p}+1)+1$, with $p>0$, or 
$m(t,r)= t^{2}(e^{t^{2} e^{r}}+1)+1$ which, {\sl en passant}, achieve their minimum $\mathfrak m=1$
in points of type $(0,r).$

\begin{remark}
It would be interesting to see what happens in case $m$ might vanish.
\end{remark}

\medskip

Let us define the map $M(t, r):=\int_{0}^{t}m(s, r)ds$; we set also
$$M_{r}:=M(\cdot, r):\R\to \R.$$
As we have mentioned before, the result of this Section provides some properties of $M_{r}$ which will play an important role in the study of problem \eqref{P}.

\begin{lemma}\label{le}
Assume \eqref{m_{0}}-\eqref{m_{1}}. Then, \smallskip
\begin{enumerate}
  \item[$(a)$]  for each $r\in [0,\infty)$ the map $M_{r}:\R \to\R $ is a strictly increasing diffeomorphism;
  \medskip
    \medskip
   \item[$(b)$] for each $r\in [0,\infty),$ the inverse map $ M^{-1}_{r}$  is  Lipschitz continuous  with Lipschitz constant $\mathfrak m^{-1}$.
     In particular $|M_{r}^{-1}(s)|\leq \mathfrak m^{-1}|s|$ for all $s\in\R$.
 \end{enumerate}
 \noindent Assume now also \eqref{m_{2}}. Then
 
 \begin{enumerate}
  \item[$(c)$]   if $s_{n}\to s_{0}$ and $r_{n}\to r_{0}$ then $M_{r_{n}}^{-1}(s_{n})\to M_{r_{0}}^{-1}(s_{0})$;
 \medskip
      \item[$(d)$] for each $r\in[0,+\infty)$, the map $s\mapsto M_{r}^{-1}(s)/s$ is (continuous and) strictly decreasing in $(0,+\infty)$ and strictly increasing  in $(-\infty, 0)$. 
\end{enumerate}
\end{lemma}
\begin{proof}
$(a)$ It is obvious. 

.

\noindent $(b)$ Fixed $r\geq0,s_{1},s_{2}\in \R$ and setting $t_{i} = M_{r}^{-1}(s_{i}), i=1,2$, we are reduced to prove
$$\mathfrak m |t_{1} - t_{2}| \leq |M_{r}(t_{1}) -M_{r}(t_{2})|.$$
Assume $t_{1}> t_{2}$. Then 
\begin{eqnarray*}
|M_{r}(t_{1} ) - M_{r}(t_{2}) | &=&M_{r}(t_{1} ) - M_{r}(t_{2}) \\ 
&=&\int_{t_{2}}^{t_{1}} m(\sigma, r)d\sigma \\
&\geq& \mathfrak m (t_{1} - t_{2}) = \mathfrak m|t_{1}- t_{2}|.
\end{eqnarray*}

\noindent $(c)$  Let us see first a simple fact.
Let $s>0$ be  fixed. By the Mean Value Theorem and \eqref{m_{2}}, for each $r\in [0,\infty)$, 
there is a unique $t_{r,s}$ between $0$ and $M_{r}^{-1}(s)$ 
such that
\begin{equation}\label{4a}
M_{r}^{-1}(s)=M_{r}^{-1}(s)-M_{r}^{-1}(0)=(M_{r}^{-1})'(M_{r}(t_{r,s}))(s-0)=\frac{s}{m(t_{r,s}, r)}.
\end{equation}
The unicity follows since $m_{r}$ is  strictly increasing in $(0,+\infty)$.
Equivalently, $t_{r,s}$ is the unique positive number satisfying
\begin{equation}\label{5a}
s=M\left(\frac{s}{m(t_{r,s}, r)}, r\right).
\end{equation}
Moreover by \eqref{m_{1}} and (\ref{4a}), it follows that 
\begin{equation}\label{6a}
t_{r,s}< M_{r}^{-1}(s)\leq \mathfrak m^{-1}s.
\end{equation} 

If $s<0$, using  that $m_{r}$ is strictly decreasing in $(-\infty,0)$ we conclude again the existence
of a unique $t_{r,s}$ between $M_{r}^{-1}(s)$ and $0$, which satisfy \eqref{5a}
and the inequalities in \eqref{6a} hold with absolute values.

\smallskip

Now, let $\{s_{n}\},\{r_{n}\}$ be  sequences  such that $r_{n}\to r_{0}\geq0, s_{n}\to s_{0}$.

If $s_{0}\neq0$, from \eqref{6a} (or \eqref{6a} with absolute values in case $s_{0}<0$), up to a subsequence, there is $t_{\ast}\in\R$ such that $t_{r_{n}, s_{n}}\to t_{\ast}$. 
Passing to the limit in $n$ in the identity
$$s_{n}=M\left(\frac{s_{n}}{m(t_{r_{n}, s_{n}}, r_{n})}, r_{n}\right)$$
(recall \eqref{5a})
 and using the continuity of $m$ and $M$, it follows that
$$
s_{0}=M\left(\frac{s_{0}}{m(t_{\ast}, r_{0})}, r_{0}\right).
$$
By the unicity we infer that $t_{\ast}=t_{r_{0},s_{0}}$. Consequently,
$$
M_{r_{0}}^{-1}(s_{0})=\frac{s_{0}}{m(t_{r_{0},s_{0}}, r_{0})},
$$
showing that
$$
M_{r_{n}}^{-1}(s_{n})=\frac{s_{n}}{m(t_{r_{n},s_{n}}, r_{n})}\longrightarrow\frac{s_{0}}{m(t_{r_{0},s_{0}}, r_{0})}=M_{r_{0}}^{-1}(s_{0}).
$$

If $s_{0}= 0,$ then from $(b)$ we have  $M_{r_{n}}^{-1}(s_{n})\to 0=M_{r_{0}}^{-1}(0)$.

%
\medskip

\noindent $(d)$ Let $r\geq0$ be fixed. Consider the case  $s_{1}>s_{2}>0$. Setting
$$ t_{i}:=M_{r}^{-1}(s_{i})>0,\quad i=1,2$$
 we are reduced to show that 
\begin{equation}\label{eq:monot}
\frac{M_{r}(t_{1})}{t_{1}}>\frac{M_{r}(t_{2})}{t_{2}}
\end{equation}
From \eqref{m_{2}} and
$$
M_{r}(t)=\int_{0}^{t}m(\sigma, r)d\sigma<m(t, r)\int_{0}^{t}d\sigma=m(t, r)t, 
$$
we deduce

$$
\left(\frac{M_{r}(t)}{t}\right)'=\frac{m(t, r)t-M_{r}(t)}{t^{2}}>0 \quad \forall  t>0,
$$
which implies \eqref{eq:monot}.

The case $s_{1}<s_{2}<0$ is treated similarly.
\end{proof}

\medskip

\section{The case $f(x, u)=f(x)$}\label{sec:f(x)}
In this section we address  the problem
\begin{equation}\label{NP}
\left \{ \begin{array}{ll}
-\text{div}\left (m(u, |\nabla u|_{2}^{2})\nabla u\right)= f(x) & \mbox{in $\Omega$,}\smallskip\\
u=0 & \mbox{on $\partial\Omega$,}
\end{array}\right.
\end{equation}
where $f\in L^{q}(\Omega), q>N/2$ and  $N\geq 2$.

We prove the following

\begin{theorem}\label{th:main1}
If \eqref{m_{0}}-\eqref{m_{2}} hold, $0\not\equiv f\in L^{q}(\Omega)$ and $q>N/2$,  
then problem \eqref{NP} has a nontrivial weak solution $u_{*}$.
\end{theorem}

Let us consider for every $r\geq0$ the problem
\begin{equation}\label{Pr}\tag{$P_{r}$}
\left \{ \begin{array}{ll}
-\text{div}\left(m_{r}(u)\nabla u\right)= f(x) & \mbox{in $\Omega$,}\smallskip \\
u=0 & \mbox{on $\partial\Omega$,}
\end{array}\right.
\end{equation}
 whose weak solution is, by definition, 
 a function $u_{r}\in H_{0}^{1}(\Omega)\cap L^{\infty}(\Omega)$ satisfying
$$
\forall \varphi\in H_{0}^{1}(\Omega):\quad \int_{\Omega}m_{r}(u_{r})\nabla u_{r}\nabla \varphi dx=\int_{\Omega}f(x)\varphi dx.
$$
 By our assumptions the equality above makes sense.

We observe that $u\in H^{1}_{0}(\Omega)\cap L^{\infty}(\Omega)$ solves
 \eqref{Pr} if and only if   $v:=M_{r}(u)\in H^{1}_{0}(\Omega)\cap L^{\infty}(\Omega)$  satisfies
\begin{equation}\label{LP}
\left \{ \begin{array}{ll}
-\Delta v= f(x) & \mbox{in $\Omega$,} \smallskip \\
v=0 & \mbox{on $\partial\Omega$,}
\end{array}\right.
\end{equation}
in a weak sense.
Note also that $v$ does not depend on $r$, due to the unicity of the solution of \eqref{LP}.
Since $f\in L^{q}(\Omega)$ with $q>N/2$, from Sobolev embeddings, the weak solution $v$ of \eqref{LP} belongs to $C(\overline{\Omega})$.
Then $u_{r}:=M_{r}^{-1}(v)$ is a weak solution of \eqref{Pr} and belongs to $C(\overline \Omega)$,
being composition of continuous functions.


The next step in the study of (\ref{NP}) is given by the next proposition.

\begin{proposition}\label{prop}
If \eqref{m_{0}}-\eqref{m_{2}} hold and $0\not\equiv f\in L^{q}(\Omega), q>N/2$, 
then for each $r\geq 0$, the auxiliary problem (\ref{Pr})  has a unique nontrivial weak solution $u_{r}$. Moreover, 
the map
 $$S: r\in[0,+\infty)\longmapsto \int_{\Omega}|\nabla u_{r}|^{2}dx\in \mathbb [0,+\infty)$$ is continuous.
\end{proposition}

\begin{proof}
As we have seen, the unicity of the solution of \eqref{Pr} is a consequence of the unicity of the solution of \eqref{LP}:
we just need to show the second statement in the Proposition.
Let $\{r_{n}\}$ be a sequence of nonnegative numbers such that $r_{n}\to r_{0}\geq0$.
Setting $u_{n}:=u_{r_{n}}$, it holds
$$
 \forall  n\in\mathbb N: \int_{\Omega}|\nabla u_{n}|^{2}dx=\int_{\Omega}\frac{1}{[m_{r_{n}}(u_{n})]^{2}}|\nabla v|^{2}dx,
$$
where   $v$ is the solution of \eqref{LP}. From Lemma \ref{le} c) and the continuity of $m$, we have pointwise in $\Omega$
\begin{eqnarray*}
u_{n}(x)=M_{r_{n}}^{-1}(v(x))&\longrightarrow &M_{r_{0}}^{-1}(v(x))=u_{r_{0}}(x) \\
\frac{1}{[m_{r_{n}}(u_{n})]^{2}}|\nabla v|^{2}&\longrightarrow& \frac{1}{[m_{r_{0}}(u_{r_{0}})]^{2}}|\nabla v|^{2} \ \ \ \mbox{ a.e. in }\Omega.
\end{eqnarray*}
On the other hand,  by \eqref{m_{1}}
\begin{equation*}\label{8}
\frac{1}{[m_{r_{n}}(u_{n})]^{2}}|\nabla v|^{2}\leq \frac{1}{\mathfrak m^{2}}|\nabla v|^{2} \  \ \ \mbox{ a.e. in } \Omega.
\end{equation*}
We deduce by 
the Dominated Convergence Theorem that
$$
\int_{\Omega}|\nabla u_{n}|^{2}dx=\int_{\Omega}\frac{1}{[m_{r_{n}}(u_{n})]^{2}}|\nabla v|^{2}dx\longrightarrow \int_{\Omega}\frac{1}{[m_{r_{0}}(u_{r_{0}})]^{2}}|\nabla v|^{2}dx=\int_{\Omega}|\nabla u_{r_{0}}|^{2}dx,
$$
concluding the proof.
\end{proof}

{\sl Proof  of Theorem \ref{th:main1}.}
The idea is to show that 
the continuous map  $S$ defined in Proposition \ref{prop}
  has a fixed point. 

Clearly, being $u_{0}$ a nontrivial solution of (\ref{Pr}) with $r=0$, it holds
\begin{equation*}\label{9}
S(0)=\int_{\Omega}|\nabla u_{0}|^{2}dx>0.
\end{equation*}
Introducing the map $T:[0,\infty)\to[0,\infty)$ given by $$T(r):=\int_{\Omega}m(u_{r}, r)|\nabla u_{r}|^{2}dx\equiv
\int_{\Omega}\frac{1}{m(u_{r}, r)}|\nabla v|^{2}dx,$$
it is 
$$
 \forall r\geq 0: \ \  S(r)\leq \frac{1}{\mathfrak m}T(r)\leq \frac{1}{\mathfrak m^{2}}\int_{\Omega}|\nabla v|^{2}dx.
$$
In particular there is $R>0$ such that
\begin{equation*}\label{10}
 \forall r\geq R:  \ \ S(r)<r.
\end{equation*}
The above considerations allow us to conclude that $S$ has a fixed point $r_{*}>0$
and then, for $u_{*} := u_{r_{*}}$, we have
$$
r_{\ast}=S(r_{\ast})=\int_{\Omega}|\nabla u_{\ast}|^{2}dx.
$$
In other words,  $u_{*}$ is  a solution of  \eqref{NP}.
\begin{remark}
Observe that we have showed that the map $S$ is bounded.
\end{remark}

\medskip

\section{The general case $f=f(x,u)$}\label{sec:general}
Throughout this section, $f:\Omega\times \R\to\R$ is a  function satisfying the following  condition
\begin{enumerate}[label=(C),ref=C,start=0]
\item\label{C} $f$ is a Carath\'eodory function and
$$\exists c>0 : \quad |f(x, t)|\leq c(1+|t|^{p}), \quad (x, t)\in \Omega\times\R, $$
where $1<p<2^{\ast}-1$ and $2^{*}=2N/(N-2).$
\end{enumerate}

As in the previous Section we begin by introducing, for fixed $r\geq0$, the auxiliary problem
\begin{equation}\label{P'r}\tag{$P'_{r}$}
\left \{ \begin{array}{ll}
-\text{div}\left(m_{r}(u)\nabla u\right)= f(x,u) & \mbox{in $\Omega$,}\smallskip \\ 
u=0 & \mbox{on $\partial\Omega$,}
\end{array}\right.
\end{equation}
associated to (\ref{P}). 

Considering the same change of variable $v=M_{r}(u)$, we see that $u$ solves \eqref{P'r}
if and only if $v$ satisfies the following semilinear problem
\begin{equation}\label{SP}\tag{$SP_{r}$}
\left \{ \begin{array}{ll}
-\Delta v= h_{r}(x, v) & \mbox{in $\Omega$,} \smallskip\\
v=0 & \mbox{on $\partial\Omega$,}
\end{array}\right.
\end{equation}
where $$h_{r}(x, t):=f(x, M_{r}^{-1}(t)).$$
Of course, a weak solution of \eqref{SP} is a function $v\in H_{0}^{1}(\Omega)\cap L^{\infty}(\Omega)$ such that
$$
 \forall  \varphi\in H_{0}^{1}(\Omega): \ \int_{\Omega}\nabla v\nabla \varphi dx=\int_{\Omega}h_{r}(x, v)\varphi dx.
$$
In view of Lemma \ref{le} $(b)$ and \eqref{C}, the right hand side above makes sense.

Now we state the criterium anticipated in the Introduction which give us a  sufficient condition 
for the existence of solutions for problem \eqref{P}. 

\begin{theorem}\label{th:th}
Let $m$ and $f$ satisfy \eqref{m_{0}}-\eqref{m_{2}} and \eqref{C}, respectively. Suppose, additionally,
that for each $r\geq 0$, problem \eqref{SP} has a unique nontrivial weak solution $v_{r}\in H^{1}_{0}(\Omega)\cap L^{\infty}(\Omega)$ and the 
map 
 $$V: r\in[0,+\infty)\longmapsto \int_{\Omega}|\nabla v_{r}|^{2}dx\in \mathbb [0,+\infty)$$ 
is in $L^{\infty}([0,+\infty))$. Then \eqref{P} possesses a nontrivial weak solution.
\end{theorem}
\begin{proof}
By assumptions, for each $r\geq 0$, the auxiliary problem \eqref{P'r} has a unique nontrivial solution 
$u_{r}=M^{-1}_{r}(v_{r})\in H^{1}_{0}(\Omega)\cap L^{\infty}(\Omega)$. 
The idea is to show that the map
 $$T: r\in[0,+\infty)\longmapsto \int_{\Omega}|\nabla u_{r}|^{2}dx\in \mathbb [0,+\infty)$$ 
has a fixed point $u_{*}$, which will be a solution of  \eqref{P}.
Let us begin by proving that  $T$ is continuous.

Let $\{r_{n}\}$ be a sequence of nonnegative numbers such that $r_{n}\to r_{0}\geq0$. Then, setting for brevity
$u_{n}:=u_{r_{n}}$ and $v_{n}:=v_{r_{n}}$, 
the boundedness of $V$ implies that 
 $\{v_{n}\}$ is bounded in $H_{0}^{1}(\Omega)$. Then
there exists $v\in H^{1}_{0}(\Omega)$, such that, up to subsequences,
\begin{equation*}\label{4}
v_{n}\rightharpoonup v \ \text{ in } H^{1}_{0}(\Omega) \quad\text{and } \quad v_{n}\to v \ \mbox{ a.e. in $\Omega$}
\end{equation*}
and, for some   $g\in L^{p}(\Omega)$,
\begin{equation*}\label{6}
|v_{n}(x)|\leq g(x) \ \mbox{ a.e. in $\Omega$}.
\end{equation*}
Consequently, passing to the limit in $n$ in the identity
$$
\int_{\Omega}\nabla v_{n}\nabla v dx=\int_{\Omega}h_{n}(x, v_{n})v dx=\int_{\Omega}f(x, M_{r_{n}}^{-1}(v_{n}))v dx,
$$
where $h_{n}:=h_{r_{n}}$, and using \eqref{C} and  Lemma \ref{le} $(b), (c)$, by the Dominated Convergence Theorem, we obtain
\begin{equation}\label{7}
\int_{\Omega}|\nabla v|^{2} dx=\int_{\Omega}f(x, M_{r_{0}}^{-1}(v))v dx.
\end{equation}
Similarly, passing to the limit in $n$ in
$$
\int_{\Omega}|\nabla v_{n}|^{2}dx=\int_{\Omega}h_{n}(x, v_{n})v_{n} dx=\int_{\Omega}f(x, M_{r_{n}}^{-1}(v_{n}))v_{n} dx,
$$
it follows that
\begin{equation}\label{8}
\int_{\Omega}|\nabla v_{n}|^{2} dx\longrightarrow \int_{\Omega}f(x, M_{r_{0}}^{-1}(v))v dx.
\end{equation}
From (\ref{7}) and (\ref{8}), we conclude that
\begin{equation*}\label{9}
\int_{\Omega}|\nabla v_{n}|^{2} dx\longrightarrow \int_{\Omega}|\nabla v|^{2} dx.
\end{equation*}
Then $v_{n} \to v$ in $H^{1}_{0}(\Omega)$.
Then up to a subsequence,
\begin{equation*}\label{10}
|\nabla v_{n}(x)|^{2}\to |\nabla v(x)|^{2} \quad \mbox{ a.e. in $\Omega$}
\end{equation*}
and  for some $\tilde{g}\in L^{1}(\Omega)$,
\begin{equation*}\label{11}
|\nabla v_{n}(x)|^{2}\leq \tilde{g}(x) \quad \mbox{ a.e. in $\Omega$},
\end{equation*}
so that, 
\begin{equation}\label{14}
\frac{1}{[m(u_{n}, r_{n})]^{2}}|\nabla v_{n}|^{2}\leq \frac{1}{\mathfrak m^{2}}\tilde{g}(x) \quad \mbox{ a.e. in $\Omega$}.
\end{equation}

Finally, from Lemma \ref{le} $(c)$, we have 
\begin{equation*}\label{12}
u_{n}(x)=M_{r_{n}}^{-1}(v_{n}(x))\longrightarrow M_{r_{0}}^{-1}(v(x)) \quad \mbox{ a.e. in $\Omega$}.
\end{equation*}

As before,
passing to the limit in $n$ in the identity
$$\int_{\Omega} \nabla v_{n} \nabla \varphi\, dx= \int_{\Omega} h_{n}(x,v_{n})\varphi \,dx =\int_{\Omega}f(x, M_{r_{n}}^{-1}(v_{n}))\varphi \,dx,$$
we have
$$\int_{\Omega} \nabla v \nabla \varphi \,dx= \int_{\Omega}f(x, M_{r_{0}}^{-1}(v))\varphi \,dx ,$$
which implies, by uniqueness of the solution, that $M_{r_{0}}^{-1}(v(x))=u_{r_{0}}(x)$
and so $u_{n} \to u_{r_{0}}$ a.e. in $\Omega$.
Then 
\begin{equation}\label{13}
\frac{1}{[m(u_{n}, r_{n})]^{2}}|\nabla v_{n}|^{2}\longrightarrow \frac{1}{[m(u_{r_{0}}, r_{0})]^{2}}|\nabla v|^{2} \quad \mbox{ a.e. in $\Omega$}
\end{equation}
and  from (\ref{13}), (\ref{14}) and the Dominated Convergence Theorem, we get
$$
\int_{\Omega}|\nabla u_{n}|^{2}dx=\int_{\Omega}\frac{1}{[m(u_{n}, r_{n})]^{2}}|\nabla v_{n}|^{2}dx\longrightarrow \int_{\Omega}\frac{1}{[m(u_{r_{0}}, r_{0})]^{2}}|\nabla v|^{2}dx=\int_{\Omega}|\nabla u_{r_{0}}|^{2}dx,
$$
proving the continuity of $T$.

\medskip

To prove that $T$ has a fixed point, we observe that $T(0)=\int_{\Omega}|\nabla u_{0}|^{2}dx>0$ and, being $r\mapsto \int_{\Omega}|\nabla v_{r}|^{2}dx$ bounded,  there exists $R>0$ such that
$$
\forall r\geq R:\quad T(r)=\int_{\Omega}\frac{1}{[m(u_{r}, r)]^{2}}|\nabla v_{r}|^{2}dx\leq \frac{1}{\mathfrak m^{2}}\int_{\Omega}|\nabla v_{r}|^{2}dx\leq r.
$$
Then the existence of a fixed point is guaranteed.
\end{proof}

\begin{remark}
Of course the main ingredients (and difficulties) in  Theorem \ref{th:th} are the existence of a unique solution
to the auxiliary problem \eqref{SP} as well as the {\sl a priori} bound of the solutions with respect to $r$.

However the hypothesis on the unicity of the solution to problem \eqref{SP} is not strictly necessary,
and actually not satisfied in many semilinear elliptic problems.
All that we need is  the existence of a fixed point for the map $T$ which can be achieved, e.g., in the following case.

Assume that, even though problem \eqref{SP} has not a unique solution, it is possible to
choose continuously in $r$ the solution $v_{r}$ of  \eqref{SP}, so that we can define a {\em continuos} branch of solutions and the map $V:r\mapsto \int_{\Omega} |\nabla v_{r} |^{2}dx$ is continuous.
In this case, being the map $r\mapsto v_{r}\in H^{1}_{0}(\Omega)$ continuous, it is easy to see, using Lemma \ref{le} (c)
that also $r\mapsto u_{r}=M_{r}^{-1}(v_{r})\in H^{1}_{0}(\Omega)$, and hence $T$, is continuous. The existence of a fixed point for $T$ 
is then guaranteed if one can prove that  $\lim_{r\to+\infty}V(r)/r < \mathfrak m^{2}.$
\end{remark}

\section{Two particular cases}\label{sec:particulares}

As we  mentioned before the goal of this section is to present some particular examples of problems covered by  Theorem \ref{th:th}. Certainly, Theorem \ref{th:th} covers a range of different situations, however, we limit ourselves to consider two special cases.

\subsection{First case}
Under assumptions  \eqref{m_{0}}-\eqref{m_{2}} on $m$, let
 $f:\Omega\times [0,+\infty)\to\R$ be a  function which  satisfies the following conditions: 

\begin{enumerate}[label=(f\arabic*),ref=f\arabic*,start=1]
\item\label{f_1} there exists $c>0$ such that 
  $$
  0\leq f(x, t)\leq c(1+t) \ \mbox{ a.e. in $\Omega$ and $\forall t\in [0,+\infty)$},
  $$ 
  \end{enumerate}
\begin{enumerate}[label=(f\arabic*),ref=f\arabic*,start=2]
\item\label{f_{2}}  for a.e. $x\in\Omega$, the function $t\mapsto f(x, t)$ is continuous on $[0,+\infty)$,  \medskip
\item\label{f_{3}}  for each $t\geq 0$, the function $x\mapsto f(x, t)$ belongs to $L^{\infty}(\Omega)$,\medskip

\item\label{f_{4}} for a.e. $x\in\Omega$, the function $t\mapsto f(x, t)/t$ is decreasing on $(0,+\infty)$.   
\end{enumerate}

\medskip

In what follows we 
  extend $f$ on the negative numbers by giving the constant value $f(x,0)$. This extension,
denoted again with $f$,
satisfies obviously \eqref{f_1}-\eqref{f_{4}}.

Recall that $h_{r}(x, t)=f(x, M_{r}^{-1}(t))$.
We prove the following 

\medskip
{\bf Claim:} $h_{r}$ satisfies \eqref{f_1}-\eqref{f_{4}}.
\medskip

Indeed, it follows from \eqref{f_{2}}, \eqref{f_{3}} and \eqref{m_{1}} (see also Lemma \ref{le} $(b)$) that 
 $h_{r}$ satisfies also \eqref{f_{2}}, \eqref{f_{3}} for all $r\geq 0$ and $t\in\R$. Moreover, from \eqref{f_1} and \eqref{m_{1}} (see Lemma \ref{le} $(b)$ again), the function $h_{r}$ verifies, for every $r\geq0$,  the growth condition
\begin{equation}\label{eq:gc}
\exists \,c>0: \quad
  |h_{r}(x, t)|\leq c(1+|t|) \ \mbox{ a.e. in }  \Omega,  \forall t\in \R
\end{equation}
that is, a \eqref{f_1}-like condition.
Moreover from \eqref{f_{4}} and Lemma \ref{le} $(a)$ and $(d)$, we have, for $t_{1}>t_{2}>0$,  a.e. in $\Omega$:
\begin{eqnarray*}
\frac{h_{r}(x, t_{1})}{t_{1}}&=&\frac{f(x, M_{r}^{-1}(t_{1}))}{M_{r}^{-1}(t_{1})}\frac{M_{r}^{-1}(t_{1})}{t_{1}}\\
&\leq&\frac{f(x, M_{r}^{-1}(t_{2}))}{M_{r}^{-1}(t_{2})}\frac{M_{r}^{-1}(t_{1})}{t_{1}}\\
&< &\frac{f(x, M_{r}^{-1}(t_{2}))}{M_{r}^{-1}(t_{2})}\frac{M_{r}^{-1}(t_{2})}{t_{2}}\\
&=&\frac{h_{r}(x, t_{2})}{t_{2}},
\end{eqnarray*}
showing that $h_{r}$ satisfies also  condition \eqref{f_{4}}, proving the claim.

\medskip

Note that the unique point in this case  where  the positivity of $f$ is used is in the above computation.

Then, according to \cite[Theorem 1, first part]{BO} there is at most one solution to \eqref{SP}. To prove the existence,
 define the functions (following \cite{BO})
\begin{eqnarray*}
\alpha_{0}(x):=\lim_{t\to 0}\frac{f(x, t)}{t} \ \ && \Big(\geq f(x,1)\Big)\\ 
\alpha_{\infty}(x):=\lim_{t\to \infty}\frac{f(x, t)}{t}\ \ && \Big(\leq f(x,1)\Big);
\end{eqnarray*}
then, as a  consequence of \eqref{f_{3}} and \eqref{f_{4}}, there is  a constant $C\geq0$ such that
\begin{equation*}\label{eq:limitacao}
\alpha_{0}(x)\in[0,+\infty] \quad \text{and} \quad \alpha_{\infty}(x)\in[0 ,C] \quad \text{for a.e. }   x\in \Omega.
\end{equation*}

At the same way it is natural to introduce, for  every $r\geq 0$, the functions
\begin{eqnarray*}
\alpha_{0}^{r}(x)&:=&\lim_{t\to 0}\frac{h_{r}(x, t)}{t} \\
 \alpha_{\infty}^{r}(x)&:=&\lim_{t\to+\infty}\frac{h_{r}(x, t)}{t} 
\end{eqnarray*}
Note that the limits exist, since $h_{r}$ satisfies \eqref{f_{4}}, with 
$$\alpha_{0}^{r}(x) \in [h_{r}(x,1), +\infty] \quad \text{ and } \quad \alpha_{\infty}^{r}(x)\in [0, h_{r}(x,1)] , \quad \text{for a.e. } \ x\in \Omega.$$
Since $h_{r}$ satisfies also  \eqref{f_{3}}, for  $C_{r}:=|h_{r}(\cdot,1)|_{\infty}$ it is actually
\begin{equation}\label{eq:limitacoes}
\alpha_{0}^{r}(x)\in [0,+\infty] \quad \text{and} \quad \alpha_{\infty}^{r}(x)\in[0, C_{r}], \quad \text{for a.e. } \ x\in \Omega.
\end{equation}
Of course it may happens $\alpha_{0}^{r}=+\infty$.

\begin{remark}
We recall (see \cite[Section 3]{BO}) that  given a measurable function $\alpha$ which is bounded above or below, we denote with
\begin{equation*}\label{1}
\lambda_{1}(-\Delta-\alpha(x))=\inf_{u\in H_{0}^{1}(\Omega)}\frac{\int_{\Omega}|\nabla u|^{2}dx-\int_{\Omega, u\neq0}\alpha(x)u^{2}dx}{\int_{\Omega} u^{2}dx}
\end{equation*}
``first eigenvalue'' of the operator $-\Delta-\alpha(x)$ with Dirichlet boundary condition. Note that it can be also
 $+\infty$ or $-\infty$. However, in our case due to \eqref{eq:limitacoes} it is
$$\lambda_{1}(-\Delta - a_{0}^{r}(x)) \in [-\infty, +\infty) \quad \text{ and } \quad \lambda_{1}(-\Delta - a_{\infty}^{r}(x)) \in (-\infty, +\infty).$$
\end{remark}

Then, if we impose  the further condition \medskip
\begin{enumerate}[label=(f\arabic*),ref=f\arabic*,start=5]
\item\label{f_5} for every $r\geq0$
$$\lambda_{1}\left(-\Delta-\alpha^{r}_{0}(x)\right)<0<\lambda_{1}\left(-\Delta-\alpha^{r}_{\infty}(x)\right)$$ 
\end{enumerate}
we can apply \cite[Theorem 1, second part]{BO} to deduce the existence of a unique 
nontrivial and nonnegative weak solution $v_{r}$ to problem \eqref{SP}, with 
$v_{r} \in W^{2, q}{(\Omega)}\cap L^{\infty}(\Omega)$, for all $q\in(1,+\infty)$.

\medskip

The main result of this section is the following
\begin{theorem}\label{th:casof(xu)}
Under the conditions \eqref{m_{0}}-\eqref{m_{2}} and \eqref{f_1}-\eqref{f_5}, problem \eqref{P} admits a nontrivial and nonnegative solution $u\in W^{2, q}{(\Omega)}\cap L^{\infty}(\Omega)$, for all $q\in(1,\infty)$. 
\end{theorem}
\begin{proof}
To prove the result, we will use Theorem \ref{th:th}. We have already shown that under our hypothesis problem 
\eqref{SP} has a unique solution $v_{r}$ for every $r\geq0$. It remains to show the boundedness of the map
$V: r\in[0,+\infty)\longmapsto \int_{\Omega}|\nabla v_{r}|^{2}dx\in \mathbb [0,+\infty)$.

It is important to note that since the right hand side of the inequality in \eqref{eq:gc} does not depend on $r$, we can conclude 
by the argument used in  \cite[page 62]{BO} to prove that the solution $u$ is in $L^{\infty}(\Omega)$,
 that there is $C>0$ (independent on $r$) such
\begin{equation*}
0\leq v_{r}\leq C, \quad \forall  r\geq 0.
\end{equation*}
From this the boundedness of the map  $V: r\mapsto \int_{\Omega}|\nabla v_{r}|^{2}dx$
easily follows; indeed
\begin{eqnarray*}
\int_{\Omega}|\nabla v_{r}|^{2}dx&=&\int_{\Omega}h_{r}(x, v_{r})v_{r}\,dx\\
&\leq& \int_{\Omega}c(1+v_{r})v_{r}\,dx\\
&\leq & c(1+C)\int_{\Omega}v_{r}\,dx\\
&\leq & C\left(\int_{\Omega}|\nabla v_{r}|^{2}\,dx\right)^{1/2}.
\end{eqnarray*}
Then, all the assumptions of Theorem \ref{th:th} are satisfied, and we deduce the existence of a solution for \eqref{P}.
\end{proof}

\subsubsection{Some comments on the verification of \eqref{f_5}}
Observe that assumption \eqref{f_5} involves the function $f$ via the change of variable, and also the parameter $r$;
then it
can be not easy  to verify. However
by assuming simpler conditions
(see examples below)
the verification of \eqref{f_5} can be simplified.

\medskip

First of all, it is possible to give the dependence of $\alpha_{0}^{r}(x)$ and $\alpha_{\infty}^{r}$ 
with respect to $r$.

 Fixed $x$, by the definition of $h_{r}$, 
$$\alpha_{0}^{r}(x)=\lim_{t\to 0}\frac{h_{r}(x, t)}{t}=
\lim_{t\to 0}\frac{f(x, M_{r}^{-1}(t))}{M_{r}^{-1}(t)}\frac{M_{r}^{-1}(t)}{t}.
$$
Denoting $t=M_{r}(s)$,  by the  L'H\^{o}spital rule we get
$
\lim_{t\to 0}\frac{M_{r}^{-1}(t)}{t}=\lim_{s\to 0}\frac{1}{m_{r}(s)}=\frac{1}{m(0, r)},
$
with  $0<1/m(0, r)\leq1/\mathfrak m$ for all $r\geq 0$. Then
\begin{equation}\label{eq:ar0}
\alpha_{0}^{r}(x)=\frac{\alpha_{0}(x)}{m(0, r)}\in [0,+\infty].
\end{equation}

 In an analogous way, if we try to estimate
\begin{equation*}\label{3}
\alpha_{\infty}^{r}(x)=\lim_{t\to +\infty}\frac{f(x, M_{r}^{-1}(t))}{M_{r}^{-1}(t)}\frac{M_{r}^{-1}(t)}{t},
\end{equation*}
since $M_{r}^{-1}(t)\to+\infty$  as $t\to+\infty$ (see Lemma \ref{le}, (a)),
 we derive again by the  L'H\^{o}spital rule that  
$$
\alpha_{\infty}^{r}(x)=\frac{\alpha_{\infty}(x)}{m(\infty, r)}, \qquad \text{ where }
\ m(\infty, r):=\lim_{s\to +\infty}m_{r}(s).
$$

Now we have the following

\begin{example}\label{ex:esemplo}
Beside the assumptions \eqref{m_{0}}-\eqref{m_{2}} and \eqref{f_1}-\eqref{f_{4}}, assume also that
\medskip
\begin{enumerate}[label=(f\arabic*),ref=f\arabic*,start=6]
\item\label{f_6} for a.e. $x\in \Omega$ the map $t\mapsto f(x, t)$ is  nondecreasing in $(0,+\infty)$.
\end{enumerate}
\medskip
Then, from Lemma \ref{le} (b),
$$
\alpha_{\infty}^{r}(x)=\lim_{t\to\infty}\frac{f(x, M_{r}^{-1}(t))}{t}\leq \mathfrak{m}^{-1}\lim_{t\to\infty}\frac{f(x, \mathfrak{m}^{-1}t)}{\mathfrak{m}^{-1}t}=\mathfrak{m}^{-1}\alpha_{\infty}(x), \quad \forall \, r\geq 0.
$$ 
Consequently,
\begin{equation}\label{5}
\lambda_{1}\left(-\Delta-\mathfrak{m}^{-1}\alpha_{\infty}(x)\right)\leq \lambda_{1}\left(-\Delta-\alpha_{\infty}^{r}(x)\right), \quad \forall \, r\geq 0.
\end{equation}

Without loss of generality let  $\mathfrak{m}=\inf_{t, r}m(t, r)$ and assume the further condition
\medskip
\begin{enumerate}[label=(m\arabic*),ref=f\arabic*,start=3]
\item\label{m_3} $m(0, r)=\mathfrak{m},  \forall \, r\geq 0$.
\end{enumerate}
\medskip
Then, by \eqref{eq:ar0}, $\alpha_{0}^{r}(x)=\mathfrak{m}^{-1}\alpha_{0}(x)$  and therefore
\begin{equation}\label{6}
\lambda_{1}\left(-\Delta-\alpha_{0}^{r}(x)\right)=\lambda_{1}\left(-\Delta-\mathfrak{m}^{-1}\alpha_{0}(x)\right), 
\quad \forall \, r\geq 0.
\end{equation}
In this case, from (\ref{5}) and (\ref{6}), we see that condition
\begin{equation}\label{eq:condicaoBO}
\lambda_{1}\left(-\Delta-\mathfrak{m}^{-1}\alpha_{0}(x)\right)<0<\lambda_{1}\left(-\Delta-\mathfrak{m}^{-1}\alpha_{\infty}(x)\right)
\end{equation}
(which does not involve $r$) implies condition \eqref{f_5}. 
We are then reduced to verify condition  \eqref{eq:condicaoBO}, as in \cite{BO},
which just involve the pure $f$.
\end{example}

\medskip


Here are two cases in which the verification of \eqref{f_5} can be simplified.

\begin{example}
If  $m(\infty, r)=+\infty$, for all $r\geq 0$
then
$\lambda_{1}\left(-\Delta-\alpha^{r}_{\infty}(x)\right)=\lambda_{1}\left(-\Delta\right)>0$; then the second inequality in the assumption 
\eqref{f_5} is automatically satisfied.  
\end{example}

\begin{example}
When $m(0, r)<1$ for all $r\geq 0$ 
then $\lambda_{1}\left(-\Delta-\alpha^{r}_{0}(x)\right)<\lambda_{1}\left(-\Delta-\alpha_{0}(x)\right)$ and the first inequality in the assumption \eqref{f_5} reduces to prove an inequality which just involve the original nonlinearity $f$,
as in \cite{BO}.
\end{example}

As we have seen, the above Theorem \ref{th:casof(xu)} is based on the fact that
problem \eqref{SP} has a unique solution,
thanks to a result of \cite{BO}.
However other simple cases in which there is the unicity of the solution at
\eqref{SP} in the sublinear case are easily found in the literature (see e.g. \cite{Lions})
then other assumptions on $f$ can be given in order to obtain a solution of \eqref{P}.

\subsection{ Second case}

Under the same assumptions \eqref{m_{0}}-\eqref{m_{2}} on $m$, let $f:\Omega\times \R\to\R$ be Carath\'eodory function satisfying:
\begin{enumerate}[label=(f\arabic*),ref=f\arabic*,start=7]
\item \label{f_7}$f(x, 0)\neq 0$, \smallskip
 \item\label{f_8} there exists $\mu\in L^{2}(\Omega), \delta\in (0,1), \nu>0$, such that 
  $$
  |f(x, t)|\leq \mu(x) + \nu |t|^{\delta }\quad \mbox{ a.e. in $\Omega$ and $\forall t\in \R$},
  $$ 
  \item\label{f_9} there is $\theta\in (0, \mathfrak m\lambda_{1})$ such
  $$
  |f(x, t_{1})-f(x, t_{2})|\leq \theta|t_{1}-t_{2}| \quad \mbox{a.e. in } \Omega,
  $$
 for all $t_{1}, t_{2}\in \R$. Hereafter $\lambda_{1}$ is the first eigenvalue of $-\Delta$ in $H^{1}_{0}(\Omega)$.
 
    \end{enumerate}

\medskip

Since $f$ satisfies \eqref{f_8},  it is clear that the same holds for $h_{r}$, for any $r\geq 0$. 
On the other hand, from \eqref{f_9} and  Lemma \ref{le} $(b)$, we conclude that 
\begin{equation}\label{20}
|h_{r}(x, s_{1})-h_{r}(x, s_{2})|\leq \frac{\theta}{\mathfrak m}|s_{1}-s_{2}| \quad \mbox{ a.e. in $\Omega$},
\end{equation}
 for all $r\geq 0$ and $s_{1}, s_{2}\in \R$.
This redly implies that problem \eqref{SP} has a unique nontrivial solution for each $r\geq 0$.
The argument is known, however we revise it here for completeness.
Defining the solution operator $S_{r}: H_{0}^{1}(\Omega)\to H_{0}^{1}(\Omega)$ which associates to each $w\in H_{0}^{1}(\Omega)$ the unique solution $v$ of problem 
\begin{equation*}
\left \{ \begin{array}{ll}
-\Delta v= h_{r}(x, w) & \mbox{in $\Omega$,}\\
v=0 & \mbox{on $\partial\Omega$,}
\end{array}\right.
\end{equation*}
it follows from (\ref{20}) and the Sobolev embedding that (hereafter $\|\cdot \|$ is the $H^{1}_{0}-$norm)
\begin{eqnarray*}
\|S_{r}(w_{1})-S_{r}(w_{2})\|^{2}&\leq & \int_{\Omega}|h_{r}(x, w_{1})-h_{r}(x, w_{2})||S_{r}(w_{1})-S_{r}(w_{2})|dx\\
&\leq &\frac{\theta}{\mathfrak m} \int_{\Omega}|w_{1}-w_{2}||S_{r}(w_{1})-S_{r}(w_{2})|dx\\
&\leq &\frac{\theta}{\mathfrak m\lambda_{1}}\|w_{1}-w_{2}\|\|S_{r}(w_{1})-S_{r}(w_{2})\|
\end{eqnarray*}
showing that $S_{r}$ is a contraction for each $r\geq 0$. 

It follows by the Banach Fixed Point Theorem that, for each $r\geq 0$, there is a unique nontrivial solution 
$v_{r}\in H^{1}_{0}(\Omega)\cap L^{\infty}(\Omega)$
to problem \eqref{SP}.

Finally from
\begin{eqnarray*}
\| v_{r}\|^{2}=\int_{\Omega}h_{r}(x, v_{r})v_{r}dx 
&\leq&  \lambda_{1}^{-1/2}\| v_{r}\|  \Big | \mu+\nu|v_{r}|^{\delta}\Big|_{2} \\
&\leq &  \lambda_{1}^{-1/2} \| v_{r}\| \left[ |\mu|_{2} + \nu |v_{r}|^{2\delta}_{2} \right]\\
&\leq&   \lambda_{1}^{-1/2} \| v_{r}\|  \left[ |\mu|_{2} + \nu |v_{r}|_{2}^{\delta}|\Omega|^{(1-\delta)/2}\right] \\
&\leq& \lambda_{1}^{-1/2} \| v_{r}\|  \left[ |\mu|_{2} + \nu  \lambda_{1}^{-\delta/2 }  \| v_{r}\| ^{\delta}|\Omega|^{(1-\delta)/2}\right] \\
&=& \lambda_{1}^{-1/2}\| v_{r}\| |\mu|_{2} +\nu \lambda_{1}^{-(\delta+1)/2} |\Omega|^{(1-\delta)/2}\|v_{r}\|^{\delta+1}.
\end{eqnarray*}
we deduce the boundedness of the map $V$.

Then by Theorem \ref{th:th}
we get
\begin{theorem}
Under the conditions \eqref{m_{0}}-\eqref{m_{2}} and \eqref{f_7}-\eqref{f_9}, problem \eqref{P} admits a nontrivial  solution $u\in W^{2, q}{(\Omega)}\cap L^{\infty}(\Omega)$, for all $q\in(1,\infty)$. 
\end{theorem}

Of course the theorem also holds if in \eqref{f_8} we allow $\delta=1$ with $\nu \in (0,\lambda_{1})$.


\end{document}